\documentclass[12pt,reqno]{amsart}
\usepackage{amsaddr}
\usepackage{amssymb,latexsym,amsmath,epsfig,amsthm,mathrsfs}
\usepackage{rotating}
\usepackage{graphicx}
\usepackage{amssymb}
\usepackage{lineno}
\usepackage{enumitem}
\usepackage{cite}
\usepackage{bm}
\usepackage[top=3cm, bottom=3cm, left=3cm, right=3cm]{geometry}
\usepackage[usenames]{color}
\usepackage[colorlinks=true,
linkcolor=blue,
filecolor=blue,
citecolor=blue]{hyperref}

\usepackage{mathtools}

\makeatletter
\DeclarePairedDelimiterX{\pmodx}[1]{(}{)}{{\operator@font mod}\mkern6mu#1}
\renewcommand{\pmod}{%
	\allowbreak
	\if@display\mkern18mu\else\mkern8mu\fi
	\pmodx
}
\makeatother

\newtheorem{theorem}{Theorem}[section]

\newtheorem{lemma}[theorem]{Lemma}
\theoremstyle{definition}

\theoremstyle{definition}
\newtheorem{remark}[theorem]{Remark}
\theoremstyle{theorem}
\newtheorem{conjecture}[theorem]{Conjecture}

\numberwithin{equation}{section}

\begin{document}
	
\title[Restricted set addition in finite abelian groups]{Restricted set addition in finite abelian groups}
	
\author[V Goswami]{Vivekanand Goswami}	\address{\em{\small Department of Mathematics, Indian Institute of Technology Bhilai, Durg – 491001, Chhattisgarh, India\\
email: vivekanandg@iitbhilai.ac.in}}

	
\author[R K Mistri]{Raj Kumar Mistri$^{*}$}
\address{\em{\small Department of Mathematics, Indian Institute of Technology Bhilai, Durg – 491001, Chhattisgarh, India\\
email: rkmistri@iitbhilai.ac.in}}
	
\thanks{$^{*}$Corresponding author}
		


\begin{abstract}
Let $A$ be a nonempty subset of finite abelian group $G$ of order $n$. For an integer $h \geq 2$, the restricted $h$-fold sumset $h^\wedge A$ is the set of all sums of $h$ distinct elements of $A$. It is known that if $G$ is a group of order $n$ and $A$ is a subset of $G$ such that $|A|$ is close to $\frac{n}{2}$, then $h^{\wedge}A = G$ under some conditions on $h$ and $n$. The constant $\frac{1}{2}$ is optimal for groups of even order but not for groups of odd order. For an integer $h \geq 4$, let $\alpha_h$ be the unique positive root of the polynomial $3^{h - 2} x^{h - 1} + x - 1$. In this paper, we show that for any $\alpha > \alpha_h$, there exists a positive integer $M_h(\alpha)$, which is determined precisely, such that for all $n > M_h(\alpha)$ with $n$ odd, if $A$ is a subset of a finite abelian group $G$ of order $n$ and if $|A| \geq \alpha n$, then $h^{\wedge} A = G$. Moreover, $\alpha_h > \alpha_{h + 1}$ for $h \geq 4$ and $\alpha_h$ approaches $\frac{1}{3}$ as $h$ increases, and the constant $\frac{1}{3}$ is optimal when the smallest prime dividing $n$ is $3$. This result extends a theorem of Tang and Wei on $4^{\wedge}A$ in the cyclic group $\mathbb{Z}_n$ to $h^{\wedge}A$ for every $h \geq 4$, and to arbitrary finite abelian groups.
\end{abstract}		
	
\maketitle

\tableofcontents
		
\section{Introduction and Main Result}
Let G be an additive abelian group of order $n$, and let $\mathbb{Z}_n$ be a cyclic group of order $n$. 
Let $h \geq 2$ be an integer, and let $A$ be a subset of $G$. The cardinality of a set $A$ is denoted by $|A|$. The \emph{$h$-fold sumset} of $A$, denoted by $hA$, is defined as
\[hA = \{a_1 + a_2 + \cdots + a_h : a_i \in A\}.\]
The \emph{restricted $h$-fold sumset} of $A$, denoted by $h^\wedge A$, is defined as
\[h^\wedge A = \{a_1 + a_2 + \cdots + a_h : a_i \in A, a_i \neq a_j \ \text{for} \ i \neq j\}.\]

Understanding the growth and extremal behavior of these sumsets is central to problems in additive combinatorics. These kind of problems have been studied extensively in the litearture (see \cite{nath, tao, freiman, mann} and the references given therein). The study of sumsets can be traced back to Cauchy \cite{cauchy}, who proved that if $A$ and $B$ are nonempty subsets of $\Bbb Z_p$, then $|A + B| \geq \min(p, |A| + |B| -1)$, where $\Bbb Z_p$ is the group of prime order $p$, and $A + B = \{a + b: a \in A, b \in B\}$. This result, later rediscovered by Davenport in 1935, is now known as the Cauchy-Davenport Theorem \cite{dav1, dav2}. An immediate consequence of this theorem is that for any nonempty subset $A \subseteq \mathbb{Z}_p$ and any positive integer $h$, the $h$-fold sumset $hA$ satisfies $|hA| \ge \min(p, h|A| - h + 1)$. While the $h$-fold sumset $hA$ has been extensively studied in the literature, relatively few results are known for the restricted $h$-fold sumset $h^{\wedge}A$. The restricted $h$-fold sumset $h^{\wedge}A$ often displays behavior that is noticeably different from the $h$-fold sumset $hA$, revealing subtler additive structures. Because of these unique characteristics, analyzing restricted $h$-fold sumset $h^{\wedge}A$ typically requires methods that are significantly different from those used for the $h$-fold sumset $hA$. In 1964, Erd\H{o}s and Heilbronn conjectured that for a subset $A$ of the cyclic group $\mathbb{Z}_p$,
\begin{equation*}
	|2^\wedge A| \geq \min\{p, 2 |A| - 3\}.
\end{equation*}

For the restricted $h$-fold sumset $h^\wedge A$ in $\mathbb{Z}_p$, a corresponding result was established by Dias da Silva and Hamidoune \cite{dias} using techniques from exterior algebra.
\begin{theorem}[{\cite[Theorem 4.1]{dias}}]  Let  $h$ and $k$ be positive integers such that $h \leq k$. Let $A \subseteq {\mathbb{Z}_p}$  be a nonempty set with $k$ elements. Then
	\[\left|h^\wedge A\right| \geq \min\{p, h k - h^{2} + 1\}.\]
\end{theorem}
This result was subsequently reproved by Alon, Nathanson, and Ruzsa \cite{alon1, alon2} using the polynomial method, a powerful tool for solving various problems in additive combinatorics. In the special case $h = 2$, this theorem corresponds to the Erd\H{o}s–Heilbronn conjecture, originally proposed by Erd\H{o}s and Heilbronn \cite{erdos1964} in 1964.

Another important problem associated with the restricted $h$-fold sumset $h^{\wedge}A$ is the following: Given a finite abelian group $G$ and an integer $h \geq 2$, how large does a set $A \subseteq G$ need to be to guarantee that the set of all sums of $h$ distinct elements covers the whole group $G$, that is, $h^{\wedge}A = G$? This problem has been investigated by various researchers.

In 1999, Gallardo, Grekos and  Pihko proved the following result on $2^\wedge A$ in \cite{gallardo1999}.
\begin{theorem}[{\cite[Lemma 3]{gallardo1999}}]	
Let $A \subseteq \mathbb{Z}_n$ such that $|A| > n/2 + 1$, then $2^\wedge A = \mathbb{Z}_n$.
\end{theorem}
	
In 2002, Gallardo \textit{et al.} obtained the following results (see \cite{gallardo2002}) in connection with possible extensions of a well-known theorem of Erd\H{o}s, Ginzburg, and Ziv. 

\begin{theorem}[{\cite[Proposition 3.4]{gallardo2002}}]\label{3^A odd}
Let $n$ be an odd positive integer. For any 
\[\alpha > \alpha_{0} = \frac{\left(\sqrt{13} - 1\right)}{6},\]
there exists 
\[N = N(\alpha) = \frac{9}{3 \alpha^{2} + \alpha - 1}\]
such that for all $n > N$ and $A \subseteq \mathbb{Z}_n$, if $|A| \geq \alpha n$, then we have 
\[3 ^\wedge A = \mathbb{Z}_n.\]
\end{theorem}
	
\begin{theorem}[{\cite[Theorem 3.1]{gallardo2002}}]\label{3^A}
For any integer $n \geq 12$, except $n = 15$, and for any subset 
$A \subseteq \mathbb{Z}_n$ such that $|A| > \frac{n}{2}$, one has 
\[3^\wedge A = \mathbb{Z}_n.\]
\end{theorem}

The following conjecture was stated in \cite{gallardo2002}.
\begin{conjecture}[{\cite[Conjecture 3.9]{gallardo2002}}]
There is a constant $c$, such that for any odd integer $n$ and for any subset $A$ of $\mathbb{Z}_n$ such that $|A| > \frac{2}{5}n + c$, one has $3^{\wedge}A = \mathbb{Z}_n$.
\end{conjecture}
	
This conjecture was proved by Lev in $2002$ by establishing the following result for $3^\wedge A$ in arbitrary finite abelian group \cite{lev2002}.

\begin{theorem}\label{lev}
Let $G$ be a finite abelian group, and let $G_0 = \{g \in G: 2g = 0\}$. Let $A$ be a subset of $G$ such that
\[|A| > \max\left\{\frac{5}{13}|G|,  120\left|G_0\right| + 355\right\}.\]
Then either $3^{\wedge}A = G$, or $A$ is contained in a coset of an index two subgroup of $G$, or $A$ is contained in a union of two cosets of an index five subgroup of $G$.
\end{theorem}	
	
In 2019, Tang and Wei \cite{tang2019} obtained the following result for $4^\wedge A$ in the cyclic group $\mathbb{Z}_n$ of odd order $n$.
\begin{theorem}[{\cite[Theroem 1.3]{tang2019}}]\label{4^A odd}
Let $n \geq 11$ be an odd positive integer. For any 
\[\alpha > \alpha_{0} = \frac{3}{486} \left(27 + \sqrt{741} \right) 
+ \frac{3}{486} \left(27 - \sqrt{741} \right),\]
there exists 
\[N = N(\alpha) = \frac{54}{9 \alpha^{3} + \alpha - 1}\]
such that for all $n > N$ and $A \subseteq \mathbb{Z}_n$, if $|A| \geq \alpha n$, then we have 
\[4^\wedge A = \mathbb{Z}_n.\]
\end{theorem}
	
In this paper, we generalize Theorem \ref{3^A odd} and Theorem \ref{4^A odd} for the restricted sumset $h^{\wedge}A$ for $h \geq 4$ in arbitrary finite abelian group $G$ of odd order $n$. More precisely, we prove the following theorem.
	
\begin{theorem}\label{h^A}
Let $h \geq 4$ be an integer, and let $p(h)$ denote the number of partitions of $h$. Let $\alpha_h $ be the unique positive root of the polynomial $3^{h - 2} x^{h - 1} + x - 1$. Then for any $\alpha > \alpha_h $ there exists 
\[M_h(\alpha) = \max \left\{\frac{3^{h - 2} (h^2 - h)}{2 \left(3^{h - 2} \alpha^{h - 1} + \alpha - 1 \right)}, \frac{12 (p(h) - 4)) (h - 4)! + (3 h - 7) (h - 4)}{6 \alpha}\right\}\]
such that for all $n > M_h(\alpha)$ with $n$ odd, if $A$ is a subset of a finite abelian group $G$ of order $n$ and if $|A| \geq \alpha n$, then
\[h^{\wedge} A = G.\]

Moreover, the roots $\alpha_h$ satisfy the following properties: 
\begin{enumerate}
  \item [(a)] $\alpha_h \in (\frac{1}{3},\frac{1}{2})$ and $\alpha_{h} > \alpha_{h + 1}$ for every $h \geq 4$,
  \item [(b)] $\lim\limits_{h \to \infty}\alpha_h = \frac{1}{3}$.
\end{enumerate}
\end{theorem}
	
We give numerical values of $\alpha_h$ and $M_h(\alpha)$ corresponding to some values of $h$ and $\alpha > \alpha_{h}$ in the following table:   

\begin{table}[h]
\centering
\begin{tabular}{|c|c|c|c|}
\hline
$h$  & $\alpha_h$ (upto $3$ decimal places)&   $\alpha$           & $M_h(\alpha)$        \\ \hline
$4$  & $0.404$   &  $0.405$             & $18807.96$           \\ \hline
$5$  & $0.388$   &  $0.389$             & $37255.68$           \\ \hline
$6$  & $0.377$   &  $0.378$             & $392935.41$          \\ \hline
$7$  & $0.370$   &  $0.371$             & $1097319.46$         \\ \hline
$8$  & $0.365$   &  $0.366$             & $2777127.76$         \\ \hline
$9$  & $0.361$   &  $0.362$             & $11349436.56$        \\ \hline
$10$ & $0.358$   &  $0.359$             & $33321849.20$        \\ \hline
$11$ & $0.356$   &  $0.357$             & $57366134.14$        \\ \hline
\end{tabular}\\[20pt]
\caption{Numerical values of  $\alpha_h$ and $M_h(\alpha)$ for some values of $h$ and $\alpha > \alpha_h$.}
\end{table}

\begin{remark}
 Theorem \ref{lev} implies that for all $n \geq 1235$ and $n$ odd, if $|A| > 0.4n$ then $3^\wedge A = G$. This result further implies that if $|A| > 0.4n + 1$ then $4^\wedge A = G$, for all $n \geq 1235$ and $n$ odd. For  $n \geq 1235$, we have  $0.4n + 1 <  0.404n$. Therefore, for $n \geq 1235$, Theorem \ref{lev} implies the $h = 4$ case of our Theorem \ref{h^A}. 
\end{remark}

It is worth mentioning a result for the restricted $h$-fold sumset $h^{\wedge}A$ which has appeared recently \cite[Theorem B]{b chen} while we were preparing the final draft of this paper. We briefly discuss this result here which was proved in connection with the $h$-critical number of a finite abelian group $G$. For a positive integer $h$, the \emph{restricted $h$-critical number} $\chi^{\wedge}(G, h)$ of the finite abelian group $G$ is defined as
\[\chi^{\wedge}(G, h) = \min\{m \in \mathbb{N}: h^{\wedge}A = G~\text{whenever}~ A \subseteq G~ \text{with}~ |A| \geq m\}.\]
For the groups of even order, the precise values $\chi^{\wedge}(G, h)$ was determined by Roth and Lempel \cite{RothLempel1992} in 1992 (see also \cite{bajnok2018} for detailed discussion on crticial numbers). Very recently, Chen and Huang \cite[Theorem B]{b chen} has proved the following theorem in case of finite abelian groups of odd order.
\begin{theorem}[{\cite[Theorem B]{b chen}}] \label{chen-huang-thm}
  Let $G$ be a finite abelian group of odd order $n$, let $p(G)$ be the smallest prime divisor $n$. Let $h$ be an integer such that $3 \leq h \leq \dfrac{n}{p(G)} - 2$. Let
\[
c(n) =
\begin{cases}
\frac{2}{5} & \text{if } 5 \mid n, \\
\frac{5}{13} & \text{if } 5 \nmid n.
\end{cases}
\]

Then the $h$-critical number satisfies the upper bound
\[
\chi^{\wedge}(G, h) \leq
\begin{cases}
\lfloor c(n) n \rfloor + 9 
& \text{if } p(G) = 3 \text{ and } n \geq 3 \cdot 46319, \\

\lfloor c(n) n \rfloor + 21 
& \text{if } p(G) = 5 \text{ and } n \geq 5 \cdot 1235, \\

\lfloor c(n) n \rfloor + 3 
& \text{if } p(G) \geq 7 \text{ and } n \geq 1235.
\end{cases}
\]
\end{theorem}

The above theorem gives an upper bound on $\chi^{\wedge}(G, h)$ which is close to $\frac{2}{5}n$ or $\frac{5}{13}n$ depending on whether $n$ is divisible by $5$. But Theorem \ref{h^A} gives more precise information about $\chi^{\wedge}(G, h)$ and implies an improved upper bound: $\chi^{\wedge}(G, h) \leq \alpha n$ for all $\alpha > \alpha_h$ under some conditions on the order of the group $n$, where $\alpha_h$ strictly decreases as $h$ increases, and approaches the limit $\frac{1}{3}$. In fact, for $h \geq 6$, the value of $\alpha_h$ becomes even less than $\frac{5}{13}$. We remark that if a set $A$ is contained in a coset of a subgroup of $G$ of index three, then $|A| \leq \frac{n}{3}$ but $h^{\wedge} A \neq G$. Thus the constant $\frac{1}{3}$ is optimal.

The organization of the paper is as follows. We fix some general notations in the next section. In Section \ref{sec-basic-concepts}, we discuss some basic concepts from group algebra and character theory of finite abelian group which will be required for the proofs. In Section \ref{sec2}, we prove auxiliary lemmas which will be reqiured for the proof of the main theorem. Finally, in Section \ref{proof}, we prove Theorem \ref{h^A}. 

\section{Notation} \label{sec-notation}
A partition of a positive integer $h$ is a $h$-tuple $(\lambda_1, \lambda_2, \dots, \lambda_h)$ such that $0 \leq \lambda_1 \leq \lambda_2 \leq \cdots \leq \lambda_h$ and $\lambda_1 + \lambda_2 +\cdots + \lambda_h = h$. Let $\mathbb{P}(h)$ denote set of all partitions of $h$, and let $p(h)$ be the number of partitions of $h$. Then
\[\mathbb{P}(h) = \{(\lambda_1, \lambda_2, \dots, \lambda_h): 0\leq \lambda_1 \leq \lambda_2 \leq \cdots \leq \lambda_h \ \text{and} \ \lambda_1 + \lambda_2 +\cdots + \lambda_h = h\}.\]
The set $\mathbb{P}(h)$ is equipped with colexicographic order ``$\prec$" defined as follows: 
\[(d_1, d_2, \dots, d_h)  \prec (e_1, e_2, \dots, e_h)\]
if and only if there exist $k \in \{1, 2, \dots, h\}$ such that $d_i = e_i$ for all $i > k$ and $d_k < e_k$. Then $|\mathbb{P}(h)| = p(h)$. Let $\boldsymbol{\lambda_i}$ denote the $i^{\mathrm{th}}$ element of $\mathbb{P}(h)$ in the colexicographic order. Then
 \[\boldsymbol{\lambda_1} = (1, 1, \dots, 1) \prec  \boldsymbol{\lambda_2}  = (0, 1, \dots, 1, 2 ) \prec \cdots \prec   \boldsymbol{\lambda_{p(h)}} = (0, 0, \dots, 0, h).\]
For $r \in \{1,2 \dots, h\}$ and $\boldsymbol{\lambda} \in \mathbb{P}(h)$, let $\mu_{\boldsymbol{\lambda}}(r)$ denotes the multiplicity of $r$ in $\boldsymbol{\lambda}$. If  $\boldsymbol{\lambda} = \boldsymbol{\lambda_i}$, then we simply write $\mu_{i}(r)$ in place of $\mu_{\boldsymbol{\lambda}}(r)$. Let  $\boldsymbol{\lambda_i} =  (\lambda_{i1}, \lambda_{i2}, \dots, \lambda_{ih})$. For $h\geq 4$, let $\boldsymbol{\lambda_{i_0}} = (0, 0, 1, \dots, 1, 3)$.
	
Let $G$ be a finite abelian group, and let  $A = \{ a_1, a_2,\dots,a_k\} \subseteq G$. Let 
\[A^h = \{(a_{j_1},a_{j_2}, \dots, a_{j_h}): a_{j_r} \in A \ \text{for}\ r = 1, 2, \dots, h\}.\]
For $m \in G$, we define 
\begin{equation}\label{R(m)}
R(m) = \left|\{(a_{j_1}, a_{j_2}, \dots, a_{j_h}) \in A^h : m = a_{j_1} + a_{j_2} + \cdots + a_{j_{h}}, a_{j_r} \neq a_{j_s} \ \text{for}\ j_r \neq j_s\}\right|
\end{equation}
Corresponding to each $\boldsymbol{\lambda_i} \in \mathbb{P}(h)$, we define  
\begin{align*}
R_i(m) = |\{(\underbrace{a_{j_1}, \ldots, a_{j_1}}_{\lambda_{i1} \ \text{times}},  
\ldots, 
\underbrace{a_{j_h}, \ldots, a_{j_h}}_{\lambda_{ih} \ \text{times}}) \in A^h: 
m = \lambda_{i1}a_{j_{1}} + \cdots + \lambda_{ih}a_{j_{h}}\}|. 
\end{align*}
Let $i_{1} = \min\{j: 1\leq j \leq h \ \text{and} \ \lambda_{ij} > 0\}$. Then $R_i(m)$ can be expressed as
\begin{equation}\label{R_i(m)}
	R_i(m) = |\{(\underbrace{a_{j_{i_1}}, \ldots, a_{j_{i_1}}}_{\lambda_{ii_1} \ \text{times}},  
	\ldots, 
	\underbrace{a_{j_h}, \ldots, a_{j_h}}_{\lambda_{ih} \ \text{times}}) \in A^h: 
	m= \lambda_{ii_1} a_{j_{i_1}} + \cdots + \lambda_{ih} a_{j_{h}}\}|
\end{equation}
for $i = 1, \ldots, p(h)$.

For any real number $v$, we denote 
\[e(v) = \exp(2 \pi i v).\] 

\section{Basic Concepts from Group Algebra and Character Theory}\label{sec-basic-concepts}
In this section, we follow the definitions and notations from \cite{grynkiewicz2013}.
\subsection{Group Algebra}
Let $\mathbb{C}$ be the field of complex numbers, and let $G$ be the finite abelian group. The group algebra $\mathbb{C}[G]$ is the all polynomials in the variable $x$ with coefficients from $\mathbb{C}$ and exponents from $G$. More precisely, $\mathbb{C}[G]$ consists of all formal expression of the form $f = \sum_{g \in G} c_g x^g$ with $c_g \in \mathbb{C}$, and multiplication and addition obeying the following rules:
\begin{enumerate}
 \item \[(ax^g)(bx^h) = abx^{g + h} ~\text{for}~ a, b \in \mathbb{C},\]
  \item \[\sum_{g \in G} a_g x^g + \sum_{g \in G} b_g x^g = \sum_{g \in G} (a_g + b_g) x^g,\]
  \item \[c\left(\sum_{g \in G} a_g x^g \right) = \sum_{g \in G} c a_g x^g,\]
  \item \[\left(\sum_{g \in G} a_g x^g\right) \left(\sum_{g \in G} b_g x^g\right) = \sum_{g \in G}\left(\sum_{h \in G} a_h b_{g - h}\right) x^g.\] 
\end{enumerate}

\subsection{Character Theory}	
Let $G$ be a finite abelian group, written additively. Let $\widehat{G} := \mathrm{Hom}(G, \mathbb{C^{\times}})$ be the group of characteres of the finite abelian group $G$, where the group operation in $\widehat{G}$ is written multiplicatively. Then $G \cong \widehat{G}$. We write $g \mapsto \chi_g$ to denote an isomorphism from $G$ onto $\widehat{G}$. Then $\widehat{G} = \{ \chi_g : g \in G \}$. This isomorphism depends on the choice of basis for $G$. Given a subset $A$ of $G$, we fix a basis $\{e_1, \dots, e_s\}$ of $G$, so that each element $a \in G$ can be expressed uniquely as $a = \sum_{j = 1}^{s} y_je_j$, with $y_j \in [0, \mathrm{ord}(e_j) - 1]$, where $\mathrm{ord}(e_j)$ denote the order of $e_j$. We use this basis to define the function $S_A : G \to \mathbb{C}$ by
\[S_A = \sum_{g \in A} \chi_g,\]
where the $\chi_g \in \widehat{G}$. If
\[g = \sum_{j = 1}^{s} r_j e_j \in G\]
and if $u_j = \mathrm{ord}(e_j)$, then $\chi_g$ is defined by
\begin{equation*}
	\chi_g \left(\sum_{j = 1}^{s} y_j e_j \right) = \prod_{j = 1}^{s} e\left(\frac{r_j y_j}{u_j} \right) = e \left(\frac{1}{u} \sum_{j = 1}^{s} r_j y_j \frac{u}{u_j}\right),
\end{equation*}
where $y_j, r_j \in \mathbb{Z}$ and $u = \mathrm{exp}(G)$, the exponent of the group $G$. The following identities will be useful:
\begin{itemize}
	\item $\chi_g(a) \chi_h(a) = \chi_{g + h}(a)$,
	
	\item $\chi_g(a + b) = \chi_g(a) \chi_g(b)$,
	
	\item $\chi_g(h) = \chi_h(g)$,
	
	\item  The image of $G$ under $\chi_g$ is a cyclic group of the $d$-th roots of unity, where $d$ is order of $g$.
\end{itemize}

\section{Auxiliary Lemmas} \label{sec2}
Let $x_1, \dots, x_k $ be variables. For $t \geq 1$, let 
\begin{equation}\label{power-sum-poly}
 p_t:= p_t(x_1,\dots, x_k) = x_1^t + \cdots + x_k^t.
\end{equation} 
For $h \geq 0$, the \emph{elementary symmetric polynomial} $e_h(x_1, \dots, x_k)$ is defined as
\begin{equation}\label{sym-poly}
e_h:= e_h(x_1, \dots, x_k) =
\begin{cases}
  1, & \mbox{if } h = 0; \\
  \sum\limits_{1 \leq j_1 < \cdots < j_h \leq k} x_{j_1}\cdots x_{j_h} , & \mbox{if } 1 \leq h \leq k; \\
  0, & \mbox{if } h > k.
\end{cases}
\end{equation}

The following lemma for symmetric polynomials is crucial for the proof of main theorem.

\begin{lemma}[{\cite[p. 6-7]{macmahon}}]\label{newton} 
For integers $t \geq 1$ and $h \geq 0$, let $p_t$ and $e_h$ be defined as in \eqref{power-sum-poly} and \eqref{sym-poly}. Then 
	\begin{equation*}
		e_h = (-1)^h \sum_{\substack{m_1 + 2 m_2 + \cdots + h m_h = h \\  m_1 \geq 0, \dots, m_h \geq 0} } \prod_{r = 1}^{h} \frac{(-p_r)^{m_r}}{m_r!r^{m_r}}.
	\end{equation*}
\end{lemma}

We prove the following lemma using Lemma \ref{newton}. 
\begin{lemma}\label{identity 1} 
Let  $A = \{a_1, a_2, \dots,a_k\}$ be a nonempty subset of $G$. Let $R(m)$ and $R_i(m)$ be defined as in \eqref{R(m)} and \eqref{R_i(m)}. Then for $m \in G$, we have the following identity: 
	\begin{equation*}
		R(m) = \sum_{i = 1}^{p(h)} (-1)^h h! \prod_{r = 1}^{h} \frac{(-1)^{\mu_i (r)}} {(\mu_i (r))! r^{\mu_i (r)}} R_i(m).
	\end{equation*}
\end{lemma}

	\begin{proof} 
Let $x$ be a variable, and let $x_j = x^{a_j}$ for $j = 1, 2, \dots, k$. Consider the polynomial $p_1:= x_1 + \cdots + x_k$ in the group algebra $\mathbb{C}[G]$. Then
	\begin{align*}
		p_1^h = (x_1 + x_2 + \cdots + x_k )^h & = (x^{a_1} + x^{a_2} + \cdots + x^{a_k} )^h \\
        & = \sum_{1 \leq j_1,\dots, j_h \leq k} x^{a_{j_1} + a_{j_2} + \cdots + a_{j_h}}\\
		& = \sum_{ m \in hA} \Biggl(\sum_{\substack{1 \leq j_1,\dots, j_h \leq k \\ a_{j_1} + a_{j_2} + \cdots + a_{j_h} = m } } x^m\Biggr)\\
		& = \sum_{ m \in hA} x^m \Biggl(\sum_{\substack{1 \leq j_1,\dots, j_h \leq k \\ a_{j_1} + a_{j_2} + \cdots + a_{j_h} = m } }1\Biggr)\\
		& = \sum_{ m \in hA} R_1(m) x^m.
	\end{align*}
Thus
\begin{equation*}
  p_1^h = \sum_{m \in hA} R_1(m) x^m,
\end{equation*}
and so it follows from Lemma \ref{newton} that
	\begin{equation}\label{representation-lem-eq1}
		\sum_{m \in hA} R_1(m) x^m = h!e_h + (-1)^{h + 1} h!\sum_{\substack{m_1 + 2 m_2 + \cdots+ h m_h = h \\  m_1 \geq 0, \dots, m_h \geq 0  \\ m_1 \neq h} } \prod_{r = 1}^{h} \frac{(-p_r)^{m_r}}{m_r! r^{m_r}}.
	\end{equation}
Any $\boldsymbol{\lambda_i} \in \mathbb{P}(h)$ can be expressed as 
\[\boldsymbol{\lambda_i} = (\underbrace{0, \ldots, 0}_{m_0 \ \text{times}}, \underbrace{1, \ldots, 1}_{m_1 \ \text{times}}, \ldots, \underbrace{h, \ldots, h}_{m_h \ \text{times}}),\]
where $m_r \geq 0$ for $r = 1, \dots, h$, and $m_0 = h - \sum_{r = 1}^{h}m_r$. Then $m_1 + 2 m_2 + \cdots + h m_h = h $. Clearly,  $m_r = \mu_i(r)$ for $r \in \{1,2,\dots, h\}$. Conversely, if $m_1 + 2 m_2 + \cdots + h m_h = h $,  then this corresponds to the partition $\boldsymbol{\lambda_i} \in \mathbb{P}(h)$ given by 
\[\boldsymbol{\lambda_i} = (\underbrace{0, \ldots, 0}_{m_0 \ \text{times}}, \underbrace{1, \ldots, 1}_{m_1 \ \text{times}}, \ldots, \underbrace{h, \ldots, h}_{m_h \ \text{times}}),\]
where $m_0 = h - \displaystyle\sum_{r = 1}^{h} m_r$. Thus there is a one-to-one correspondence between the set of partitions $\mathbb{P}(h)$ and the set of $h$-tuples $(m_1, \ldots, m_h)$ of nonnegative integers satisfying $m_1 + 2 m_2 + \cdots + h m_h = h$. Also $m_1 = h$ corresponds to $\boldsymbol{\lambda_1} = (1, 1, \dots, 1)$. Therefore, the identity \eqref{representation-lem-eq1} can be expressed as follows:

\begin{equation*}
		p_1^h = h! e_h +  (-1)^{h+1} h!\sum_{\substack{\boldsymbol{\lambda} \in \mathbb{P}(h) \\ \boldsymbol{\lambda} \neq \boldsymbol{\lambda_1}}} \Biggl(\prod_{r = 1}^{h} \frac{(-p_r)^{\mu_{\boldsymbol{\lambda}} (r)}} {(\mu_{\boldsymbol{\lambda}} (r))! r^{\mu_{\boldsymbol{\lambda}} (r)}}\Biggr).
	\end{equation*}
The above identity can be written as
	\begin{equation}\label{equation 3.1}
		p_1^h = h! e_h +  (-1)^{h + 1} h!\sum_{i = 2}^{p(h)} \Biggl(\prod_{r = 1}^{h} \frac{(-p_r)^{\mu_i (r)}} {(\mu_i (r))! r^{\mu_i (r)}}\Biggr).
	\end{equation}
For $\boldsymbol{\lambda_i} =  (\lambda_{i1}, \lambda_{i2}, \dots, \lambda_{ih}) \in \mathbb{P}(h)$, let $i_{1} = \min\{j: 1 \leq j \leq h \ \text{and} \ \lambda_{i j} > 0\}$. That is, $\lambda_{i i_1}$ is the smallest positive part of the partition $\boldsymbol{\lambda_i}$. Then it is easy to see that
\[\prod_{r = 1}^{h} (p_r)^{\mu_i (r)} = \prod_{j = i_{1}}^{h} p_{\lambda_{i j}}.\]
Hence it follows from \eqref{equation 3.1} that

\begin{align}\label{equation 3.1a}
\sum_{m \in hA} R_1(m) x^m = h! e_h + (-1)^{h + 1} h! \sum_{i = 2}^{p(h)} \Biggl(\biggl(\prod_{r = 1}^{h} \frac{(-1)^{\mu_i (r)}}{(\mu_i (r))!r^{\mu_i (r)}}\biggr)\biggl(\prod_{\substack{j = i_{1}}}^{h} p_{\lambda_{ij}}\biggr)\Biggr).
	\end{align}

Now,
	\begin{align*}
	 h! e_h =  h!\sum\limits_{1 \leq j_1 < \cdots < j_h \leq k} x_{j_1}\cdots x_{j_h} &= h! \sum_{1 \leq j_1 < j_2 < \cdots < j_h \leq k} x^{a_{j_1}} x^{a_{j_2}} \cdots x^{a_{j_h}}\\
		&= \sum_{\substack{1 \leq j_1, \dots, j_h \leq k \\ a_{j_r} \neq a_{j_s} \ \text{for} \ j_r \neq j_s }} x^{a_{j_1} + a_{j_2} + \cdots + a_{j_h}}\\
		&= \sum_{m \in hA} \Biggl(\sum_{\substack{1 \leq j_{i_1}, \dots, j_h \leq k \\ a_{j_1} + a_{j_2} + \cdots + a_{j_h} = m \\ a_{j_r} \neq a_{j_s} \ \text{for} \ j_r \neq j_s }} x^{m}\Biggr)\\ 
 &= \sum_{m \in hA}  \Biggl(\sum_{\substack{1 \leq j_1, \dots, j_h \leq k \\a_{j_1} + a_{j_2} + \cdots + a_{j_h} = m \\a_{j_r} \neq a_{j_s} \ \text{for}\ j_r \neq j_s\\ }}1\Biggr)x^m,
	\end{align*}
and so it follows from the definition of $R(m)$ that
\begin{equation}\label{equation 3.1b}
   h! e_h = \sum_{m \in hA} R(m) x^m.
\end{equation}
Next,
	\begin{align*}
	 \prod_{j = i_1}^{h} p_{\lambda_{ij}} = \prod_{j = i_1}^{h} (x_1^{\lambda_{ij}} + \cdots + x_k^{\lambda_{ij}}) &= \prod_{j = i_1}^{h} (x^{\lambda_{ij}a_1} + \cdots + x^{\lambda_{ij}a_k})\\
&= \sum_{1 \leq j_{i_{1}}, \dots, j_h \leq k} x^{\lambda_{i i_1} a_{j_{i_{1}}}}\cdots x^{\lambda_{i h} a_{j_h}}\\	
		&= \sum_{1 \leq j_{i_1}, \dots, j_h \leq k} x^{\lambda_{i i_1} a_{j_{i_1}} + \cdots + \lambda_{i h} a_{j_h}}\\
		&= \sum_{m \in hA} \biggl(\sum_{\substack{1 \leq j_{i_1}, \dots, j_h \leq k \\ \lambda_{i i_1} a_{j_{i_1}} + \cdots + \lambda_{i h} a_{j_{h}} = m } } x^{m}\biggr)\\
&= \sum_{m \in hA} \biggl(\sum_{\substack{1 \leq j_{i_1}, \dots, j_h \leq k \\ \lambda_{i i_1} a_{j_{i_1}} + \cdots + \lambda_{i h} a_{j_{h}} = m }} 1\biggr) x^{m},
	\end{align*}
and so it follows from the definition of $R_i(m)$ that

\begin{equation}\label{equation 3.1c}
 \prod_{j = i_1}^{h} p_{\lambda_{ij}} = \sum_{m \in hA} R_i(m) x^{m}.
\end{equation}
Therefore, it follows from \eqref{equation 3.1a}, \eqref{equation 3.1b} and \eqref{equation 3.1c} that

	\begin{equation*}
		\begin{aligned}
			\sum_{m \in hA} R_1(m) x^m = &\sum_{m \in hA} R(m) x^m  + (-1)^{h + 1} h! \sum_{i = 2}^{p(h)} \Biggl(\prod_{r = 1}^{h} \frac{(-1)^{\mu_i (r)}} {(\mu_i (r))! r^{\mu_i (r)}} \biggl(\sum_{m \in hA} R_i(m) x^{m}\biggr)\Biggr),
		\end{aligned}
	\end{equation*}
which implies that
\begin{equation*}
		\begin{aligned}
			\sum_{m \in hA} R_1(m) x^m = &\sum_{m \in hA} R(m) x^m  + (-1)^{h + 1} h! \sum_{m \in hA}\Biggl(\sum_{i = 2}^{p(h)} \biggl(\prod_{r = 1}^{h} \frac{(-1)^{\mu_i (r)}} {(\mu_i (r))! r^{\mu_i (r)}} R_i(m)\biggr)\Biggr) x^{m},
		\end{aligned}
	\end{equation*}
and so
	\begin{equation}\label{equation 3.2}
		\begin{aligned}
			\sum_{m \in hA} R_1(m) x^m = \sum_{m \in hA} \Biggl(R(m) + (-1)^{h + 1} h! \sum_{i = 2}^{p(h)} \biggl(\prod_{r = 1}^{h}  \frac{(-1)^{\mu_i (r)}} {(\mu_i (r))! r^{\mu_i (r)}} R_i(m)\biggr)\Biggr) x^{m}
		\end{aligned}
	\end{equation}
By comparing the coefficients of $x^m$ in \eqref{equation 3.2}, we get
	\begin{equation*}
		R_1(m) = R(m) + (-1)^{h + 1} h!\sum_{i = 2}^{p(h)} \Biggl(\prod_{r = 1}^{h} \frac{(-1)^{\mu_i (r)}} {(\mu_i (r))! r^{\mu_i (r)}} R_i(m)\Biggr).
	\end{equation*}
This implies that
	\begin{equation*}
		R(m) = \sum_{i = 1}^{p(h)} (-1)^h h! \prod_{r = 1}^{h} \frac{(-1)^{\mu_i (r)}} {(\mu_i (r))! r^{\mu_i (r)}} R_i(m).
	\end{equation*}	
\end{proof}

\begin{lemma}[{\cite[Lemma 3.3]{gallardo2002}}]\label{max sum}
Let $d \geq 3$ be an odd integer, and let $Y$ be a positive real number. For any $\mathbf{y} = (y_1, \dots, y_d) \in \mathbb{R}^d$, let 
\[T(\mathbf{y}) = \sum_{j = 1}^{d} y_j e{\left(j/d\right)}.\]
Then 
\[\max_{\mathbf{y} \in [0, Y]^d} |T(\mathbf{y})| = \frac{Y}{2\sin(\pi/2d)}.\]
\end{lemma}
	
We prove the following lemma using Lemma \ref{max sum}. The argument of the proof is similar to the argument in the proof of a similar result for a cyclic group $\mathbb{Z}_n$ in \cite[Lemma $2.5$]{tang2019}.
 
\begin{lemma}\label{max ablian}
Let $G$ be a finite abelian group of odd order $n$, where $n \geq 3$. Let $A$ be a nonempty subset of $G$, and let $g \in G$ with $g \neq 0$. Then
\[ |S_A(g)| \leq \frac{n}{3}.\]
\end{lemma}
	
\begin{proof}
Let $\chi_g \in \widehat{G}$, where $\chi_g \neq \chi_0$. The image of $G$ under $\chi_g$ is a cyclic group of the $d$-th roots of unity, where $d$ is order of $\chi_g$ in $\widehat{G}$. Since $n$ is odd and $d$ divides $n$, it follows that $d$ must be an odd integer with $d \ge 3$. We partition $G$ into $d$  cosets $C_1, \ldots, C_d$ of the kernel of $\chi_g$. These cosets are defined as
\[C_j = \{ a \in G : \chi_g(a) = e{\left(j/d\right)}\} ~\text{for}~ j = 1, \dots, d.\]
	Let
\[y_j = \left|\{ a \in A : \chi_g(a) = e{\left(j/d\right)} \}\right| ~\text{for}~  j = 1, 2, \dots, d.\]
Then
\[\sum_{a \in A} \chi_g(a)  = \sum_{j = 1}^{d} y_j e{\left(j/d\right)}.\]
Since $|C_j| = \frac{n}{d}$, it follows that $y_j = |A \cap C_j| \leq \frac{n}{d}$.
Therefore, it follows from Lemma \ref{max sum} that
\[\left| \sum_{a \in A} \chi_g(a) \right| \leq \frac{n/d} {2 \sin(\pi/2d)}.\]
We know that if $\theta \in (0, \pi/6]$, then $\sin \theta \geq \frac{3 \theta}{\pi}$. Since $\pi/2d \in (0, \pi/6]$, it follows that
\[\sin \left(\frac{\pi}{2 d}\right) \geq \frac{3 (\pi/2 d)}{\pi} = \frac{3}{2d}.\]
Hence
\[\left|\sum_{a \in A} \chi_g(a) \right| \leq \frac{n}{3}.\]
Since $\chi_g(a) = \chi_a(g)$, it follows that
\[\left|\sum_{a \in A} \chi_a(g) \right| \leq \frac{n}{3},\]
and so
\[|S_A(g)|  \leq \frac{n}{3}.\]
This completes the proof.
\end{proof}
	
\begin{lemma}[{\cite[Proposition 1.3]{grynkiewicz2013}}]\label{cha-identity 1}
	Let $G$ be a finite abelian group. Let $\chi \in \widehat{G}$. Then
		\[
		\frac{1}{|G|} \sum_{g \in G} \chi(g) =
		\begin{cases}
			1, & \chi  = \chi_0 \\
			0, & \chi  \neq \chi_0.
		\end{cases}
		\]
\end{lemma}		
	
\begin{lemma}[{\cite[Corollary 19.1.4]{grynkiewicz2013}}]\label{cha-identity 2}
 Let $G$ be a finite abelian group and let $A$ be a nonempty subset of $G$. Then 
 \begin{equation*}
 	\sum_{g \in G \setminus \{0\}} |S_A(g)|^2 = |A| |G| - |A|^2
 \end{equation*}
\end{lemma}	

\begin{lemma}\label{bound R}
Let  $A = \{a_1, a_2, \dots,a_k\}$ be a nonempty subset of a finite abelian group $G$ of odd order. Let $h$ be an integer with $h \geq 4$, and let $i_0$ be the index such that $\boldsymbol{\lambda_{i_0}} = (0, 0, 1, \dots, 1, 3)$. Let $R(m)$ and $R_i(m)$ be defined as in \eqref{R(m)} and \eqref{R_i(m)}. Then for $m \in G$, the following inequalities hold:
\begin{enumerate}
\item $R_2(m) \leq k(k-1)\cdots (k - h + 3) + \binom{h - 2}{2} R_3(m)$,
\item $R_3(m) \leq k^{h - 3}$,
\item $R_i(m) \leq k^{h - 4}$ for $i \geq 4$ with $i \neq i_0$.
\item For $i \geq 4$ with $i \neq i_0$, let $\boldsymbol{\lambda_i} =  (\lambda_{i1}, \lambda_{i2}, \dots, \lambda_{ih}) \in \mathbb{P}(h)$. Then 
\[\prod_{r = 1}^{h} (\mu_i(r))! r^{\mu_i(r)} \geq 4.\]
\end{enumerate}
\end{lemma}
	
\begin{proof}		
Note that
\begin{align*}
R_2(m) = \left|\{(a_{j_2}, \dots, a_{j_{h-1}}, a_{j_h}, a_{j_h}) \in A^h : m= a_{j_2} + \cdots + a_{j_{h - 1}} + 2 a_{j_{h}}\}\right|,
\end{align*}
and
\begin{equation*}	
  R_3(m) = \left|\{(a_{j_3}, \dots, a_{j_{h - 2}}, a_{j_{h - 1}}, a_{j_{h - 1}}, a_{j_h}, a_{j_h}) \in A^h: m= a_{j_3} + \cdots + 2 a_{j_{h-1}} + 2 a_{j_{h}}\}\right|.
\end{equation*}

To get the upper bound for $R_2(m)$, we consider two possiblities:
\begin{itemize}
  \item  If $a_{j_2}, \dots, a_{j_{h - 2}}$ all are distinct, then the number of choices for $a_{j_2}, \dots, a_{j_{h - 1}}$ is $k (k - 1) \cdots (k - h + 3)$, and any choice of $a_{j_2}, \dots, a_{j_{h - 1}}$ fixes $2 a_{j_{h}}$, and so it fixes the value of $a_{j_{h}}$, since $n$ is odd. Therefore, \[R_2(m) \leq k (k - 1) \cdots (k - h + 3).\]
  \item If at least two of  $a_{j_2}, \dots, a_{j_{h - 2}}$ are same, then clearly  $R_2(m) \leq \binom{h - 2}{2} R_3(m)$.
\end{itemize} 
Therefore,
\[R_2(m) \leq k(k-1)\cdots (k - h + 3) + \binom{h - 2}{2} R_3(m).\]

Next, since the number of choices for $a_{j_3}, \dots, a_{j_{h - 1}}$ is $k^{h - 3}$, any of these choices fixes the value of $a_{j_{h}}$, it follows that \[R_3(m) \leq k^{h - 3}.\]

A similar combinatorial argument proves that if $i \geq 4$ and $i \neq i_0$, then
\[R_i(m) \leq k^{h - 4}.\]

Finally, to prove the last inequality, we consider three cases for $\lambda_{i h}$:
\begin{itemize}
  \item [(i)] If $\lambda_{i h} < 3$, then $\lambda_{i h} = \lambda_{i(h - 1)} =  \lambda_{i(h - 2)} = 2$;
  \item [(ii)] If $\lambda_{i h} = 3$, then $ \lambda_{i(h - 1)}\geq 2$;
  \item [(iii)]  $\lambda_{i h} > 3$,
\end{itemize}
In each of the three cases, it is easy to verify that $\prod_{r = 1}^{h}(\mu_i(r))! r^{\mu_i(r)} \geq 4$. This completes the proof.
\end{proof}
	
\begin{lemma}\label{bound k}
Let $h$ and $k$ be integers such that $h \geq 4$ and $k \geq \frac{h (h - 1)}{2}$. Then the following inequality holds:
\begin{multline*}
  k (k - 1)\cdots (k - h + 3) \leq k^{h - 2} - \frac{(h - 2)(h - 3)}{2} k^{h - 3} 
  + \frac{(h - 2)(h - 3)(h - 4)(3 h - 7)}{24} k^{h - 4}.
\end{multline*}
\end{lemma}	
	
\begin{proof}
The proof is trivial for $h = 4$ and $h = 5$. Now assume that $h \geq 6$. Note that
\begin{align*}
k (k - 1)\cdots (k - h + 3) &= k^{h - 2} + \sum_{r = 1}^{h - 3}(-1)^{r} \left(\sum_{0 \leq j_1 < \cdots < j_r \leq h - 3} j_1 \cdots j_r \right) k^{h - 2 - r},
\end{align*}
which implies that

\begin{multline}\label{bound k-eq1}
k (k - 1)\cdots (k - h + 3) \\ =  k^{h - 2} - \frac{(h - 2)(h - 3)}{2} k^{h - 3}  + \frac{(h - 2)(h - 3)(h - 4)(3h - 7)}{24} k^{h - 4}\\
+ \sum_{r = 3}^{h - 3}(-1)^{r} \left(\sum_{0 \leq j_1 < \cdots < j_r \leq h - 3} j_1 \cdots j_r \right) k^{h - 2 - r}.
\end{multline}
Now, since $k \geq \frac{h (h - 1)}{2}$, it follows that if $r \geq 3$, then

\begin{align*}
\sum_{0 \leq j_1 < \cdots < j_{r + 1} \leq h - 3} j_1 \cdots j_{r + 1} &\leq \sum_{0 \leq j_1 < \cdots < j_{r} \leq h - 3} j_1 \cdots j_{r} \left(\sum_{0 \leq j_{r + 1} \leq h - 3} j_{r + 1}\right) \notag \\
&= \frac{(h - 2)(h - 3)}{2} \sum_{0 \leq j_1 < \cdots < j_{r} \leq h - 3} j_1 \cdots j_{r} \notag \\
& \leq k \sum_{0 \leq j_1 < \cdots < j_{r} \leq h - 3} j_1 \cdots j_{r}, 
\end{align*}
and so
\begin{align}\label{bound k-eq2}
	  \left(\sum_{0 \leq j_1 < \cdots < j_{r} \leq h - 3} j_1 \cdots j_{r}\right) k^{h-2-r} - \left(\sum_{0 \leq j_1 < \cdots < j_{r + 1} \leq h - 3} j_1 \cdots j_{r + 1}\right) k^{h-3-r} \geq 0.
\end{align}
Let $f(r)$ denote the left hand side of \eqref{bound k-eq2}. Now if $h$ is odd, then it follows from \eqref{bound k-eq1} that
\begin{align*}
  \sum_{r = 3}^{h - 3}(-1)^{r} \left(\sum_{0 \leq j_1 < \cdots < j_r \leq h - 3} j_1 \cdots j_r \right) k^{h - 2 - r} = - \left(\sum_{s = 3}^{\frac{h-3}{2}}f(2s-1)\right) \leq 0, 
\end{align*}
and so it follows from \eqref{bound k-eq1} that
\begin{multline*}
  k (k - 1)\cdots (k - h + 3) \leq k^{h - 2} - \frac{(h - 2)(h - 3)}{2} k^{h - 3} 
  + \frac{(h - 2)(h - 3)(h - 4)(3 h - 7)}{24} k^{h - 4}.
\end{multline*}
If $h$ is even, then it follows from \eqref{bound k-eq1} that
\begin{multline*}
  \sum_{r = 3}^{h - 3}(-1)^{r} \left(\sum_{0 \leq j_1 < \cdots < j_r \leq h - 3} j_1 \cdots j_r \right) k^{h - 2 - r} \\
   = - \left(\sum_{s = 3}^{\frac{h-4}{2}}f(2s-1)\right) -  \left(\sum_{0 \leq j_1 < \cdots < j_{h-3} \leq h - 3} j_1 \cdots j_{h-3} \right) k \leq 0, 
\end{multline*}
and so it follows from \eqref{bound k-eq1} that
\begin{multline*}
  k (k - 1)\cdots (k - h + 3) \leq k^{h - 2} - \frac{(h - 2)(h - 3)}{2} k^{h - 3} 
  + \frac{(h - 2)(h - 3)(h - 4)(3 h - 7)}{24} k^{h - 4}.
\end{multline*}
This completes the proof.
\end{proof}
	
\section{Proof of Theorem \ref{h^A}}\label{proof}
	
\begin{proof}[Proof of Theorem \ref{h^A}] 
Let  $A = \{a_1, a_2, \dots, a_k\} \subseteq G$. It follows from the Lemma \ref{identity 1} that
\begin{multline*}\label{equation 4.1}
	R(m) = R_1(m) - \frac{h^2 - h}{2} R_2(m)+  \frac{h!}{8(h - 4)!} R_3(m) \\
	+ \sum_{i = 4}^{p(h)} (-1)^h h!\prod_{r = 1}^{h} \frac{(-1)^{\mu_i (r)}} {(\mu_i (r))! r^{\mu_i (r)}} R_i(m).
\end{multline*}
We write
\begin{equation}\label{equation 4.4}
	R(m) = R_1(m) + S_1 + S_2, 
\end{equation}
where
\begin{equation*}\label{equation 4.2}
	S_1 = -\frac{h^2 - h}{2} R_2(m)+  \frac{h!}{8(h - 4)!} R_3(m),
\end{equation*}
and 
\begin{equation}\label{equation 4.3}
	S_2 = \sum_{i = 4}^{p(h)} (-1)^h h!\prod_{r = 1}^{h} \frac{(-1)^{\mu_i (r)}} {(\mu_i (r))! r^{\mu_i (r)}} R_i(m).
\end{equation}
We derive the lower bounds for $R_1(m), S_1$ and $S_2$ separately. First we estimate  $R_1(m)$. If $m \in G$, then using Lemma \ref{cha-identity 1}, we can write
\begin{align*}
R_1(m)  = \sum_{\substack{a_{j_1}, a_{j_2}, \dots, a_{j_h} \in A \\ m = a_{j_1} + a_{j_2} + \cdots + a_{j_h}}} 1 & = \sum_{a_{j_1} \in A} \sum_{a_{j_2} \in A} \dots \sum_{a_{j_h} \in A} \frac{1}{n} \sum_{g \in G} \chi_{-m + a_{j_1} + a_{j_2} + \cdots + a_{j_h}}(g) \\
& = \frac{1}{n} \sum_{g \in G} S_A(g)^h \chi_{-m}(g)\\
& \geq \frac{k^h}{n} - \frac{1}{n} \sum_{g \in G \setminus\{0\}} \left|S_A(g)\right|^h\\
& \geq \frac{k^h}{n} - \left(\max_{\substack{ g \in G \setminus\{0\}}} \left|S_A(g)\right|\right)^{h-2}
\left(\frac{1}{n} \sum_{g \in G \setminus\{0\}} \left|S_A(g)\right|^2\right),
\end{align*}
and so it follows from Lemma \ref{cha-identity 2} that
\[R_1(m) \geq \frac{k^h}{n} - \left(\max_{\substack{ g \in G \setminus\{0\}}} \left|S_A(g)\right|\right)^{h-2} \left( k - \frac{k^2}{n}\right).\]
Now it follows from Lemma \ref{max ablian} that
\begin{equation}\label{equation 4.5}
R_1(m) \geq \frac{k^h}{n} - \frac{n^{h - 2}}{3^{h - 2}} \left(k - \frac{k^2}{n}\right) = \frac{k}{3^{h - 2}n} \left(3^{h - 2}k^{h - 1} - n^{h - 1} + n^{h - 2} k\right) . 
\end{equation}
		
Now we estimate the value of $S_1$. It follows from Lemma \ref{bound R} that 
\begin{align*}
S_1 &= - \frac{h^2 - h}{2} R_2(m) + \frac{h!}{8(h - 4)!} R_3(m) \\ & \geq -  \frac{h^2 - h}{2} \left(k (k - 1) \cdots (k - h + 3) + \binom{h - 2}{2} R_3(m)\right) +  \frac{h!}{8 (h - 4)!} R_3(m)\\
&= - \frac{h^2 - h}{2}\left(k (k - 1) \cdots (k - h + 3)\right) - \frac{h!}{8 (h - 4)!} R_3(m).
\end{align*}
Since $\alpha_h < \alpha \leq 1$, it follows that
\[k \geq \alpha n > \alpha\frac{3^{h-2}(h^2-h)}{2 \left(3^{h - 2} \alpha^{h - 1} + \alpha - 1\right)} \geq \frac{(h^2 - h)}{2} = \frac{h(h - 1)}{2}.\]
Therefore, an application of Lemma \ref{bound k} gives
\[S_1 \geq -\frac{h^2 - h}{2} k^{h - 2} + \frac{h!}{4 (h - 4)!} k^{h - 3} - \frac{h!(h - 4)(3 h - 7)}{48 (h - 4)!} k^{h - 4} - \frac{h!}{8 (h - 4)!} k^{h - 3}.\]
By simplifying the above expression, we get
\begin{equation}\label{equation 4.6}
S_1 \geq - \frac{h^2 - h}{2} k^{h - 2} + \frac{h!}{8 (h - 4)!} k^{h - 3} - \frac{h! (h - 4)(3 h - 7)}{48 (h - 4)!} k^{h - 4}.
\end{equation}
Note that in \eqref{equation 4.3}, coefficient of $R_{i_0}(m)$ is positive, and by Lemma \ref{bound R}, we have $\prod_{r = 1}^{h}(\mu_i (r))! r^{\mu_i (r)} \geq 4$ for $i \geq 4$ and $i \neq i_0.$ Therefore, we get 
\begin{align*}
S_2 = \sum_{i = 4}^{p(h)} (-1)^h h! \prod_{r = 1}^{h} \frac{(-1)^{\mu_i (r)}} {(\mu_i (r))! r^{\mu_i (r)}} R_i(m) \geq \sum_{\substack{ i = 4 \\ i \neq i_0}}^{p(h)}(-1)^h h!\prod_{r = 1}^{h} \frac{(-1)^{\mu_i (r)}} {(\mu_i (r))! r^{\mu_i (r)}} R_i(m),
\end{align*}
and so
\begin{equation}\label{equation 4.7}
	S_2 \geq - \frac{h!(p(h) - 4)}{4} k^{h - 4}.
\end{equation}
Now using the inequalities \eqref{equation 4.5}--\eqref{equation 4.7} in \eqref{equation 4.4}, we get
\begin{multline}\label{equation 4.7a}
R(m) \geq \frac{k^h}{n} -\frac{n^{h - 2}}{3^{h - 2}} \left(k -\frac{k^2}{n}\right) - \left(\frac{h^2 - h}{2}\right) k^{h - 2} \\
+ \left(\frac{h!}{8 (h - 4)!} \right)k^{h - 3} - \frac{h!(h - 4) (3 h - 7)}{48 (h - 4)!} k^{h - 4} - \frac{(p(h) - 4) h!}{4} k^{h - 4}.
\end{multline}
Since  
\[k \geq \alpha n\] 
and
\[ n > \max \left\{\frac{3^{h - 2} (h^2 - h)}{2 \left(3^{h - 2} \alpha^{h - 1} + \alpha - 1 \right)}, \frac{ 12 (p(h) - 4) ((h - 4)!) + (3 h - 7) (h - 4)}{6 \alpha}\right\},\]
it follows that
\[k > \max \left\{\alpha \frac{3^{h - 2} (h^2 - h)}{2 \left(3^{h - 2}\alpha^{h - 1} + \alpha - 1 \right)}, \frac{12 (p(h) - 4) (h - 4)! + (3 h - 7) (h - 4)}{6} \right\}.\]
Therefore,
\begin{align*}
	\frac{k}{3^{h - 2} n} \left(3^{h - 2} k^{h - 1} - n^{h - 1} + n^{h - 2} k\right) & \geq \frac{\alpha n}{3^{h - 2} n} \left(3^{h - 2} \alpha^{h - 1} n^{h - 1} - n^{h - 1} + n^{h - 2} \alpha n\right)\\   
	& \geq \frac{\alpha n^{h - 1}}{3^{h - 2}} \left(3^{h - 2} \alpha^{h - 1} - 1 + \alpha\right)\\
	& > n^{h - 2} \left(\frac{h^2 - h}{2}\right)\\
	& \geq k^{h - 2}\left(\frac{h^2 - h}{2}\right).
\end{align*}
Thus 
\begin{equation}\label{equation 4.8}
\frac{k^h}{n} - \frac{n^{h - 2}}{3^{h - 2}} \left(k - \frac{k^2}{n}\right) - \left(\frac{h^2 - h}{2}\right) k^{h - 2} > 0.
\end{equation}
Also, since		
\[k \geq \alpha n \geq \frac{12 (p(h) - 4) (h - 4)! + (3 h - 7) (h - 4)}{6},\]
it follows that
\begin{equation}\label{equation 4.9}
\left(\frac{h!}{8 (h - 4)!} \right) k^{h - 3} - \frac{h! (h - 4) (3 h - 7)}{48 (h - 4)!} k^{h - 4} - \frac{(p(h)- 4) h!}{4} k^{h - 4} \geq 0.
\end{equation}
Now using the inequalities \eqref{equation 4.8} and \eqref{equation 4.9} in \eqref{equation 4.7a}, we get
\begin{align*}
	R(m) > 0,
\end{align*}
and so
\[h^{\wedge} A = G.\]

Next we prove the properties of $\alpha_h$. Let  $f_h(x) = 3^{h-2} x^{\,h-1} + x - 1$. The function $f_h(x)$ is strictly increasing for $x > 0$. Since $f_h\left(\frac{1}{3}\right) < 0$ and $f_h\left(\frac{1}{2}\right) > 0$, it follows that $f_h(x)$ has a unique positive root $\alpha_h$ that lies in the interval $\left(\frac{1}{3}, \frac{1}{2}\right)$. Since $f_h(\alpha_h) = 0$, we have $\alpha_h - 1 = -3^{h-2} \alpha_h^{h-1}$. Since $\alpha_h > \frac{1}{3}$, it follows that $f_{h+1}(\alpha_h) > 0$. Thus
\[f_{h+1}(\alpha_{h+1}) = 0 < f_{h+1}(\alpha_h).\]
Since $f_{h+1}(x)$ is strictly increasing for $x > 0$, it follows that $\alpha_{h+1} < \alpha_h$ for $h \geq 4$.
It is easy to show that
\[\lim_{h \to \infty} \alpha_h = \frac{1}{3}.\]
This completes the proof.
\end{proof}

	
\section*{Acknowledgment}
The research of the first named author is supported by the UGC Fellowship (NTA Ref. No.: 231610040283).

\end{document}